\documentclass[centertags,12pt]{amsart}

\usepackage{amssymb}
\usepackage{latexsym}
\usepackage{amsfonts}
\usepackage{amssymb}
\usepackage{graphicx} 
\usepackage{pdflscape} 
\usepackage{diagbox} 
\usepackage{color}
\usepackage{hyperref}
\usepackage[utf8]{inputenc}

\makeatletter
\renewcommand{\subsection}{\@startsection
{subsection}{2}{0mm}{\baselineskip}{-0.25cm}
{\normalfont\normalsize\em}}
\makeatother

\newtheorem{theorem}{Theorem}[section]
\newtheorem{proposition}[theorem]{Proposition}
\newtheorem{corollary}[theorem]{Corollary}
\newtheorem{lemma}[theorem]{Lemma}

{\theoremstyle{definition}

\newtheorem{example}[theorem]{Example}
\newtheorem{question}[theorem]{Question}
}

{\theoremstyle{remark}
\newtheorem{remark}[theorem]{Remark}

\newtheorem{claim}{Claim}
}

\def\1{\mathbf 1}

\def\Z{\mathbf Z}

\def\N {\mathbb{N}}
\def\Z {\mathbb{Z}}

\def\a {\alpha}

\def\k {\kappa}
\def\l {\ell}

\title[A certain sequence on  pure $\k-$sparse gapsets]{A certain sequence on  pure $\k-$sparse gapsets}

\author[G. B. Almeida Filho]{Gilberto B. Almeida Filho}
\address{Universidade Federal do Mato Grosso, Campus Várzea Grande, Mato Grosso, MT, Brazil}
\email{gilberto.filho@ufmt.br}

\author[Stéfani C. Vieira]{ Stéfani C. Vieira}
\address{Universidade Federal do Mato Grosso, Campus Cuiabá, Mato Grosso, MT, Brazil}
\email{stefani.vieira1@ufmt.br}

\thanks{{\em 2020 Math. Subj. Class.}: Primary 20M14; 
Secondary 05A15, 05A19}


\begin{document}

\begin{abstract} 
In this paper, we study the pure $\k-$sparse gapsets and our focus on getting information about the sequence observed in Table 3 at \cite{GM}, this sequence is listed in OEIS as A374773. We verify that the cardinality of the set of gapsets with genus $3n+1$ such that the maximum distance between two consecutive elements is $2n$ is equal to the cardinality of the set of gapsets with genus $3n+2$ such that the maximum distance between two consecutive elements is $2n+1$, for all $n\in \mathbb{N}$. In particular, we compute the cardinality of the symmetric and pseudo-symmetric gapsets in these cases.

\end{abstract}

 \maketitle

\section{Introduction}
\label{intro}

A \textit{gapset} is a finite set $G \subset \N$ satisfying the following property: Let $z \in G$ and write $z = x + y$, with $x$ and $y \in \N$, then $x \in G$ or $y \in G$. This concept was formally introduced by Eliahou and Fromentin \cite{EF}. We can associated a gapsets with set of gaps of a numerical semigroups. Recall that a numerical semigroup $S$ is a submonoid of the set of non-negative integers $\N_0$, equipped with the usual addition, where $G(S):=\N_0\setminus S$, the set of {\em gaps} of $S$, is finite. The cardinality of a gapset $G$ is called the {\em genus} of $G$. The set of all gapsets is denoted by $\Gamma$ and the set of all gapsets with a fixed genus $g$ by $\Gamma(g)$.

If $G$ is a gapset, we can use several known invariants of $G$. Among them are the multiplicity $m(G) := \min\{s \in \N: s \notin G\}$, the conductor $c(G) := \min\{s \in \N: s + n \notin G, \forall n \in \N_0\}$, the depth $q(G) := \lceil \frac{c(G)}{m(G)} \rceil$ and the Frobenius number of a gapset $F(G) := c(G) - 1$. In particular, if the genus of $G$ is $g$ we denote $F(G) $ by $\ell_g$.
Furthermore, the invariant genus implies a importante inclusion about gapsets. If $G$ is a non-empty gapset with genus $g$ then $G \subseteq [1,2g-1]$. Throughout this paper, we denote $[a,b] := \{x \in \Z: a \leq x \leq b\}$ and $[a, \infty) := \{x \in \Z: x \geq a\}$, for integers $a$ and $b$.

We denote by $G:=\{\ell_1<\ell_2<..<\ell_g\}$ a gapset with genus $g$. We say that $G$ is \textit{$\k$-sparse} if the differences between two consecutive elements of $G$ (in the natural order) is at most $\k$. If there are two consecutive elements $\l_{i}$ and $\l_{i+1} \in G$ such that $\l_{i+1} - \l_i = \k$, then we say that $G$ is a \textit{pure $\k$-sparse gapset}. By convention, we say that $\emptyset$ is a $0$-sparse gapset and $\{1\}$ is a $1$-sparse gapset. Our goal is the study of the following families of sets: Let $\mathcal{G}_{\k}$ be the set of all pure $\k$-sparse gapsets, the set of pure $\k$-sparse gapsets with genus $g$, namely $\mathcal{G}_{\k}(g) := \{G \in \Gamma(g): G \in \mathcal{G}_{\k}\}$ and its subset $\mathcal{G}_{\k}(g, q) := \{G \in \mathcal{G}_{\k}(g): q(G) = q\}$. 

We can use a notion of $m-$extension in the study about gapsets. The notion  of $m$-extension was introduced by Eliahou and Fromentin in \cite{EF}, for $m \in \N$. An $m$-extension $A \subset \N$ is a finite set containing $[1,m-1]$ that admits a partition \linebreak $A = A_0 \cup A_1 \cup \ldots \cup A_t$, for some $t \in \N_0$, where $A_0 = [1,m-1]$ and $A_{i+1} \subseteq m + A_i$ for all $i$. In particular, if $A$ is an $m$-extension, then $A \cap m\N = \emptyset$. Almeida  and Bernandini introduced more general sets and they define them as follows: let $m$ be a positive integer. We say that $M \subset \N$ is a $m$-set if $[1, m-1] \subset M$ and $M \cap m\N = \emptyset$. In particular, an $m$-extension is a $m$-set. We denote by $\mathcal{M}_\k(g)$ the set of all $m$-sets with $g$ elements, such that the maximum distance between two consecutive elements (with respect to the natural order) is $\k$ and that lies on $[1,2g - 1]$. In some cases, we also deal with subsets of $[1, 2g - 1]$ with no specific property. We denote by $\mathcal{C}_{\k}(g)$ the set of those sets that have $g$ elements and such that the maximum distance between two consecutive elements (with respect to the natural order) is $\k$. Notice that $\mathcal{G}_\k(g) \subset \mathcal{M}_\k(g) \subset \mathcal{C}_\k(g)$.

Some authors studied the set of numerical semigroups with fixed genus and some other invariants. For instance, Bernardini and Almeida  \cite{GM} studied numerical gapsets pure $\k$-sparse; Blanco and Rosales \cite{BR} studied numerical semigroups with fixed genus and fixed Frobenius number; Bras-Amor\' os \cite{Amoros3} studied numerical semigroups with fixed genus and fixed ordinarization number; Kaplan \cite{Kaplan} studied numerical semigroups with fixed genus and fixed multiplicity; all of them obtained some partial interesting results. The study of gapsets (or numerical semigroups) has been generated some interesting numerical sequences. For example, the cardinality of $\Gamma(g)$ is denoted by $n_g$ has first few elements of the  $1, 1, 2, 4, 7, 12, 23, 39, 67, 118$ and it is registered as the sequence A007323 at OEIS (the on-line encyclopedia of integer sequences). This sequence was deeply studied after Bras-Amor\' os \cite{Amoros1} conjectured some statements, among them, that ``$n_g + n_{g+1} \leq n_{g+2}$, for all $g$?'', which is still an open question. In the paper \cite{GM} which studied pure $\k-$sparse gapsets with genus $g$ satisfying the condition $2g\leq 3\k$. Under this condition, the results obtained by Almeida  and Bernanrdini generated the sequence $(g_w)$ whose first terms are: $1, 2, 5, 12, 30, 70, 167, 395, 936$ and $2212$,  registered as the sequence A348619  at OEIS.

The sequence studied in this paper was noted in the paper \cite{GM} in Table 3 which was created to illustrate the behavior of the sequence $(g_w)$, reproduced here in Table \ref{g,k}. Thus, we can see the existence of a sequence of $\k-$pure sparse gapsets with genus $g$ that do not satisfy the condition $2g\leq 3\k$. This sequence was registered as the sequence A374773 at OEIS, whose first terms are: 3, 8, 22, 54, 135, 331 and 808. We illustrate the sequence this sequence in Table \ref{tab1}.

Here is an outline of this paper. In section \ref{sec2}, we present some general properties of gapsets. In section \ref{sec3},   we study pure $(2n)-$sparse  gapsets with genus $3n+1$ and check some properties regarding the elements of these gapsets. In this case, we show some important properties about symmetric gapsets and verify the non-existence of pseudo-symmetric gapsets under these conditions. In particular, we calculate the number of symmetric gapsets. In section \ref{sec4}, we study pure $(2n+1)-$sparse gapsets with genus $3n+2$ and check some properties regarding the elements of these gapsets. In this case, we show some important properties about pseudo-symmetric gapsets and verify the non-existence of symmetric gapsets under these conditions. In particular, we calculate the number of pseudo-symmetric gapsets. In section \ref{sec5}, we construct a map $\sigma$ from the family $\mathcal{G}_{2n}(g)$ to the set of subsets of $[1,2g + 1]$. In particular, we prove that the map takes gapsets with depth at most 3 into non pseudo-ymmetric gapsets with depth at most 3 and, therefore, establishes a bijection between the set of pure $(2n)-$sparse  gapsets with genus $3n+1$ into the set of pure $(2n+1)-$sparse gapsets with genus $3n+2$. Thus, we conclude that $\#\mathcal{G}_{2n}(g) = \#\mathcal{G}_{2n+1}(g + 1)$, for all $n\in \mathbb{N}$ and $g=3n+1$. In section \ref{sec6}, we discuss some further questions related to this work.

\section{Preliminaries}\label{sec2}
We bring some familiar results on numerical semigroup theory to gapset theory and
we keep the proofs for the sake of completeness. We say that a gapset $G$ of genus $g$ is
\begin{itemize}
    \item \textit{symmetric gapset} if $F(G)=2g-1$;
    \item \textit{pseudo-symmetric gapset} if $F(G)=2g-2$.    
\end{itemize}

Let $G$ be a gapset. We say that an integer $x$ is a \textit{pseudo-Frobenius number} if $x\in G$ and $ x+s$ is a non-gap for all non-gap $s\neq 0$. We will denote by $PF(G)$ the set of pseudo-Frobenius numbers of $G$, and its cardinality is the \textit{type} of $G$, denoted by $\tau(G)$. In particular, $x\in G$ is a non pseudo-Frobenius number if exists a non-gap $s\neq 0$ such that $x+s\in G$.



The next result guarantees that every $m$-extension with depth at most 2 is a gapset.

\begin{proposition}\label{q=2}
Let $G \subseteq [1, 2m-1]$, with $[1,m-1] \subseteq G$ and $m \notin G$. Then $G$ is a gapset with multiplicity $m$ and depth at most $2$.
\label{q2}
\end{proposition}

\begin{proof}
Let $z \in G$ and write $z = x + y$, with $x \leq y$. Since $z \leq 2m-1$, then $x \leq m-1$ and we conclude that $x \in G$; hence $G$ is a gapset. Moreover, the conductor of $G$ is at most $2m - 1$; thus, $q(G) \leq 2$.
\end{proof}

\begin{proposition}
If a non-empty gapset has genus $g$ and multiplicity $m$, then $2 \leq m \leq g+1$.
\label{mg}
\end{proposition}

\begin{proof}
If $m \geq g+2$, then $[1,g+1] \subseteq G$ and $\# G \geq g+1$, which does not occur. Hence $m \leq g+1$. On the other hand, since $G \neq \emptyset$, then $1 \in G$ and $m \geq 2$.
\end{proof}

\begin{example}
Let $G = \{\l_1 < \l_2 < \cdots < \l_g\}$ be a pure $1$-sparse gapset. Then $\l_{i+1} - \l_i = 1$, for all $i$. Hence, $G = \{1, 2, \ldots, g\}$ is an ordinary gapset. Note that this is the only pure $1-$sparse gapset.
\end{example}

Now we obtain a relation between the multiplicity of a gapset and the maximum distance between to consecutive elements.

\begin{proposition}
Let $G$ be a pure $\k$-sparse gapset with multiplicity $m$. Then $\k \leq m$.
\label{km}
\end{proposition}

\begin{proof}
Let $\l_i$ and $\l_{i+1} \in G$ such that $\l_{i+1} - \l_i = \k$. If $\k > m$, then $\l_i + 1, \l_i + 2, \ldots, \l_i + m \notin G$ ($\l_i + m < l_{i+1}$). There is $a \in \Z$ such that $\l_i + 1 \leq \l_{i+1} - am \leq \l_i + m$.  Notice that $G \ni \l_{i+1} = (\l_i - am) + am$. But $\l_i - am \notin G$ and $am \notin G$, which leads to a contradiction.
\end{proof}

\begin{proposition}
Let $G = \{\l_1 < \l_2 < \ldots < \l_g\}$ be gapset with multiplicity $m$ and genus $g$. Then $[am + \l_j + 1, am + \l_{j+1} - 1] \cap G = \emptyset$ for all $a \in \N_0$ and for all $j \in [1,g-1]$.
\label{gilberto_partition}
\end{proposition}

The next result is important for the proofs some results of this paper. It confirms that the Frobenius number of a gapset cannot be too big with respect to its
element $\ell_\a$.

\begin{proposition}
Let $G = \{\l_1 < \l_2 < \cdots < \l_\a < \l_{\a+1} = \l_\a + \k < \cdots < \l_g\}$ be a pure $\k$-sparse gapset with genus $g$. Then $\l_g \leq \l_\a + m$.
\label{abaixo}
\end{proposition}
\begin{proof}
Suppose that $\l_g > \l_\a + m$ and let $r$ be the smallest index such that $\l_r > \l_\a + m$. By Proposition \ref{gilberto_partition}, we conclude that $\l_r \geq \l_{\a+1} + m$ and thus $\l_r - \l_{r-1} \geq \k$. Since $G$ is a pure $\k$-sparse gapset, we have $\l_r - \l_{r-1} = \k$. Notice that $r \geq \alpha + 1$. In fact, we know that $\l_j \leq \l_{\a} < \l_{\a} + m$, for all $j \in [1,\a]$. By the definition of $\a$, we know that $\a$ is the greatest index such that $\l_{i+1} - \l_i = \k$ and thus $r \leq \a +1$. Hence, the only possibility is $r = \a + 1$. However $\l_{\a+1} = \l_\a + \k \leq \l_\a + m$ and we obtain a contradiction. Hence, $\l_g \leq \l_\a + m$.
\end{proof}


The next results are important for the proofs of the main results of this paper. From [Corollary 4.11, \cite{Rosales}] and [Corollary 4.11, \cite{Rosales}]  the following two results holds.

\begin{theorem}\label{Corollary 4.11}
Let G be a numerical semigroup. The following conditions are equivalent. 
\begin{enumerate}
    \item  G is symmetric.
    \item $PF(G) = \{F(G)\}$.
    
\end{enumerate}

\end{theorem}

\begin{theorem}\label{Corollary 4.16}
 Let S be a numerical semigroup. The following conditions are
equivalent. 
\begin{enumerate}
 \item  G is pseudo-symmetric.
    \item $PF(G) = \{F(G), \frac{F(G)}{2}\}$.\end{enumerate}
\end{theorem}


    
    

    

    
    

    
    

    



\begin{proposition}[ Corollary 2, \cite{GM} ]\label{posicao do salto}
    Let $G = \{\ell_1<\ell_2<\cdots<\ell_\alpha<\cdots< \ell_g \}$ be a pure
$k-$sparse gapset with genus g, where $\alpha = max\{i : \ell_{i+1} - \ell_i = k\}$, and consider its canonical partition $G_0 \cup G_1 \cup \cdots\cup G_{q-1}$. Then one of the following occurs:
\begin{itemize}
    \item $\ell_{\alpha},\ell_{\alpha+1}\in G_{q-2}$
    \item $\ell_{\alpha},\ell_{\alpha+1}\in G_{q-1}$
    \item $\ell_{\alpha}\in G_{q-2}$ and $\ell_{\alpha+1}\in G_{q-1}$
\end{itemize}
\end{proposition}
\begin{proof}
From Proposition \ref{abaixo}, we conclude that $\l_\a \in G_{q-2}$ or  $\l_\a \in G_{q-1}$, because $\l_g \in G_{q-1}$. For the first case, both possibilities can occur: $\l_{\a+1} \in G_{q-2}$ or $\l_{\a+1} \in G_{q-1}$. For the second case, we must have $\l_{\a+1} \in G_{q-1}$ and we are done.
\end{proof}

\begin{theorem}\label{cota para tipo de G}
 Let $G $ be  gapset with genus g, multiplicity $m$ and $q$, then  $ G_{q-1}\subset PF(G) $. In partcular, $ \# G_{q-1}\leq \tau(G)$.
\end{theorem}
\begin{proof}
For  $x\in G_{q-1}$ we have $x\geq (q-1)m+1$, since $G_{q-1}\subset [(q-1)m+1, qm-1] $ and $G=G_0 \cup G_1 \cup \cdots\cup G_{q-1}$. Given $s\neq 0$ a non-gap we obtain that $x+s\geq (q-1)m+1 +m=qm+1$. Then, $x+s$ is a non-gap for all non-gap $s\neq 0$, i.e., $x\in PF(G)$.

\end{proof}

\section{On pure $(2n)-$sparse  gapsets with genus $3n+1$  }\label{sec3}

In this section we let consider the family of the pure $\k-$sparse gapsets with genus $g$ with $g=3n+1$ and $\k=2n$, for all $n\in \mathbb{N}$, denoted by $ \mathcal{G}_{2n}(g)$. More precisely, we  study the structure of this gapsets and some properties. In particular, we calculate the quantity of the symmetric gapsets in $ \mathcal{G}_{2n}(g)$.


Throughout the section,   $n$ denote a  positive integer and $g=3n+1$.

\begin{lemma}Let $n$ be a  positive integer and $g=3n+1$.
    The family $ \mathcal{G}_{2n}(g)$ has a hyperelliptic gapset if, and only if, $n=1$. In particular, $G=\{1,3,5,7\}\in  \mathcal{G}_{2}(4)$.
\end{lemma}
\begin{proof}
    The fact, if some family $ \mathcal{G}_{2n}(3n+1)$ has a hyperelliptic gapset $G$ we have that $2=m(G)$. The Proposition \ref{km} ensures that $ 2n\leq m$ we have  $ 2n\leq m(G)=2$, then $n=1.$
\end{proof}

The next proposition tells us that the set $\{i : \ell_{i+1} -\ell_i = 2n\}$ has an unique element.

\begin{proposition}\label{unique g}     Let $G $ be a pure $(2n)-$sparse gapset with genus $g$. Then
there is an unique $ \alpha\in [1,g-1] $ such that $ \ell_{\alpha+1} -\ell_\alpha=2n   $,  for all $n\in \mathbb{N}$ with $n> 2$.
\end{proposition}
\begin{proof}
    Suppose there are distinct $ \alpha, \beta\in [1,g-1]$  such that $  \ell_{\alpha+1} -\ell_\alpha=2n$ and  $\ell_{\beta+1} -\ell_\beta=2n$. Then, in $[\ell_{\alpha},\ell_{\alpha+1}]\cup [\ell_{\beta},\ell_{\beta+1}]$ we have $4n-2$ non-gaps.  Then, in $[1,2g-1]$ we have $4n-2$ non-gaps and $3n+1$ gaps. Therefore, $7n-1=4n-2 +3n+1\leq\# [1,2g]=2g-1=6n+1$ what is a contradiction.
\end{proof}

    

The next result gives us information about the multiplicity fo a symmetric gapset with genus $3n+1$ and pure $(2n)-$sparse.

\begin{proposition}\label{m=2n}
     Let $G $ be a pure $(2n)-$sparse  gapset with genus $g=3n+1$ and multiplicity $m$. If $G$ is symmetric, then $m=2n$.
\end{proposition}
\begin{proof}
Suppose that $G$ symmetric.    By Theorem \ref{Corollary 4.11} we have $PF(G)=\{\ell_g\}$, so for each $x\in G$ exists a non-gap $s\neq 0$ such that $x+s\in G$, for all $x\neq \ell_g$. Taking $x=\ell_{g-1}$ exists a non-gap $s\neq 0$ such that $\ell_{g-1}+s\in G$, then $\ell_{g-1}+s=\ell_{g}$, i.e., $ \ell_{g}-\ell_{g-1}=s$. By Proposition \ref{km} we have $ 2n\leq m$ and conclude that $$2n\geq \ell_{g}-\ell_{g-1}=s\geq m\geq 2n.$$

\end{proof}

\begin{corollary} 
     Let $G $ be a pure $(2n)-$sparse gapset with genus $g=3n+1$ and depth $q$. Then $ q\leq 4$. 
\end{corollary}
\begin{proof}
Since $\ell_g \leq 2g-1=6n+1$ and $2n\leq m$ by Proposition \ref{km}, we have that $$q=\left\lceil  \frac{c(G)}{m} \right\rceil\leq \left\lceil \frac{6n+2}{2n}\right\rceil=4.  $$
\end{proof}

\begin{theorem}\label{Gsimetrico equiv q=4}
Let $n$ be a  positive integer and $g=3n+1$.    Let $G $ be a pure $(2n)-$sparse gapset with genus $g$ and depth $q$. Then  $G$ is symmetric if, and only if, $q=4$.
\end{theorem}
\begin{proof}
If $G$ is symmetric, then $\ell_g = 2g-1=6n+1$ and $q=\left\lceil  \frac{c(G)}{m} \right\rceil= \left\lceil \frac{6n+2}{2n}\right\rceil=4.  $

Reciprocally, if $q=4$ we have by Proposition \ref{posicao do salto} we have $ \ell_\alpha \geq 2m+1$, where $\alpha = max\{i : \ell_{i+1} - \ell_i = 2n\}$. By Proposition \ref{km} we have 
$$ \ell_g\geq \ell_{\alpha+1}=2n+ \ell_{\alpha}\geq 2m+1+2n\geq 3(2n)+1=2g-1.$$
Therefore, $\ell_g=2g-1$ and $G$ is symmetric. 
\end{proof}

By straightforward calculation, the following holds.

\begin{proposition}
    For each $n\in \mathbb{N}$ with $n>1$ exists a pure $(2n)-$sparse gapset with genus $g=3n+1$ such that $q(G)=4$.
\end{proposition}
\begin{proof}
    Take $G=[1,2n-1]\cup \{2n+1\}\cup [3n+1,4n-1]\cup \{4n+1,6n+1\}$.
\end{proof}

 For the next result we consider a pure $(2n)$-sparse gapset $G$ with genus $3n+1$. If $q(G) \leq 3 $, then $G=G_0\cup G_1\cup G_2$ is the canonical
 partition. The next result ensures that \(\ell_\alpha\) belongs to either \(G_0\) or \(G_1\).
\begin{proposition}\label{ell alfa 2m-1}
    Consider
$G = \{\ell_1 < \ell_2 < \cdots < \ell_\alpha < \ell_{\alpha+1} < \cdots < \ell_g\}$ be a pure $(2n)-$sparse gapset with genus $g$ multiplicity m and depth $q\leq 3$, where $ \alpha = max\{i :\ell_{i+1} -\ell_{i} = 2n\}$. Then $\ell_\alpha  \leq 2m - 1$.
\end{proposition}
\begin{proof}
    Suppose that $\ell_\alpha \geq 2m+1$. Then, $ \ell_{\alpha+1}=\ell_\alpha+2n\geq 2m+1+2n\geq 3(2n)+1=2g-1$. By Theorem \ref{Gsimetrico equiv q=4}  we would have $q=4$, which is a contradiction.
\end{proof}


 

The next result guarantees that the family $ \mathcal{G}_{2n}(g)$ has no any pseudo-symmetric gapsets.

\begin{theorem}\label{nao tem pseudo simetrico}
If $G$ is a pure $(2n)-$sparse  gapset with genus $g$, then  $G$  is not pseudo-symmetric.
\end{theorem}
\begin{proof} Suppose that exists a gapset $G $  pseudo-symmetric.
 By Theorem \ref{Gsimetrico equiv q=4}   we conclud that $ q:=q(G)\leq 3$ and  by Theorem  \ref{Corollary 4.16} we have $PF(G)=\{\ell_{g}, \frac{\ell_{g}}{2}\}$.
\begin{enumerate}
    \item If $q=2$, then $G=G_0 \cup G_1 $ then $x\in PF(G)$ for all $x\in G_1$. Therefore,  $\#G_1\leq \tau(G)=2$.
\begin{itemize}
 \item  If $\#G_1=0$ then $G=[1,m-1] $. So $G$ is ordinary and $2n=1$ implying an absurd.
    \item  If $\#G_1=2$ then $G=[1,m-1] \cup \{\ell_{g-1},\ell_g\} $. So $ 3n+1=g=(m-1) +2$ we have $m=3n=\frac{\ell_{g}}{2}$ what is an absurd.
    \item If $\#G_1=1$ we have $G=[1,m-1] \cup \{\ell_g\} $. So $ 3n+1=g=(m-1) +1$ and we conclude that $m-1=3n$ and $ \ell_{g-1}=m-1$. Therefore, we would have $ \ell_{g}-\ell_{g-1}=3n>2n$. Absurd.
\end{itemize}
    \item  If $q=3$, then $G=G_0 \cup G_1 \cup G_2$. Notice that  $\ell_g=2g-2=6n$, $ \frac{\ell_{g}}{2}=3n$ and $\ell_g-\frac{\ell_{g}}{2}=3n>2n$ then they must not be consecutive gaps, so $\frac{\ell_{g}}{2}<\ell_{g-1}<\ell_g$ and we conclude that $\ell_{g-1}\not\in PF(G)$. Then exists a non-gap $s\neq 0$ such that $ \ell_{g-1}+s=\ell_{g}$. So we conclude that $s=m=2n$ and therefore $$\ell_{g-1}=\ell_{g}-2n=6n-2n=4n=2m$$ what is a contradiction.
\end{enumerate}


\end{proof}

\begin{example} For $n=1$ we have $g=3n+1=4$.  The every pure $2$-sparse gapsets with genus $4$ are:
\begin{itemize}
    \item   $ \{1, 2, 3, 5 \}$,
    \item   $ \{1, 2, 4, 5 \} $,
    \item   $ \{1, 3, 5, 7 \} $.
    
\end{itemize}    
\end{example}

In the next result, we provide a  characterization for all symmetric gapsets in $ \mathcal{G}_{2n}(3n+1)$.

\begin{proposition}\label{caracterizando G simetrico q=4}
    Let $G $ be a pure $(2n)-$sparse  gapset with genus $g=3n+1$ and symmetric.  Then, $G=G_0\cup G_1\cup G_2\cup G_3$, $G_3=\{\ell_g\}$ and $G_2=\{\ell_{g-1}\}$. In particular, $\alpha=g-1$, $ \ell_{g-1}=2m+1 $, $ \ell_{g}= 3m+1$  and $\#G_1=n$.
\end{proposition}
\begin{proof}
Since $G $ is symmetric we know that $\ell_g=2g-1=3m+1$ by Proposition \ref{m=2n} and  $PF(G)=\{\ell_g\}$ by Theorem \ref{Corollary 4.11}.  The Theorem \ref{Gsimetrico equiv q=4} ensures that $G=G_0\cup G_1\cup G_2\cup G_3$. As $\ell_{g-1}$ not is pseudo-Frobenius, exists a non-gap $s\neq 0$ such that $ \ell_{g-1}+s=\ell_g$, then $$ s=\ell_g -\ell_{g-1}\leq 2n=m\leq s$$
Therefore, $\ell_{g-1}=\ell_g-2n=6n+1-2n=2m+1. $ So we conclude that $G_3=\{\ell_g\}$, $G_2=\{\ell_{g-1}\}$. 

Consequently $3n+1=\#G=2n-1 +\#G_1+2$ where we obtain $\#G_1=n$.

\end{proof}

As a consequence of the previous result, we note that $m+1$ is always a gap. We show this in the next result.

\begin{corollary}\label{m+1 esta em G}
     Let $G $ be a pure $(2n)-$sparse  gapset with genus $g=3n+1$ and symmetric. Then $m+1\in G$.
\end{corollary}
\begin{proof}
    Suppose that $m+1\not\in G$, so $m+(m+1)=2m+1$ is a non-gap. But, on the other hand by Proposition \ref{caracterizando G simetrico q=4} we have $2m+1=\ell_{g-1}$ what is a contradiction.
\end{proof}

\begin{example}
The gapset  $  G=[ 1, 2, 3, 4, 5, 6, 7, 9, 10, 11, 12, 17, 25 ]  $ has genus 13 and is a pure $8-$sparse, $\ell_g=25=3m+1$ and $m+1=9\in G$. Therefore, $G$ is a symmetric gapset with genus $g=3n+1$, where $n=4$.
\end{example}


    

\begin{theorem}\label{contagem G simetrico q=4}
     The quantity of symmetric gapsets in $\mathcal{G}_{2n}(3n+1)$ is $2^{n-1}.$
\end{theorem}
\begin{proof}
 Let $G $ be a pure $(2n)-$sparse  gapset with genus $g=3n+1$  and symmetric.
   As $G $ is symmetric we know that $\ell_g=2g-1$, by Proposition \ref{m=2n} we have $m=2n$ and $PF(G)=\{\ell_g\}$ by Theorem \ref{Corollary 4.11}. By Proposition \ref{caracterizando G simetrico q=4} we have that $\ell_g=3m+1$, $\ell_{g-1}=2m+1 $ and $G=[1,m-1]\cup G_1\cup \{2m+1\}\cup \{3m+1\}$ with $\#G_1=n$. By Corollary \ref{m+1 esta em G} we have $m+1\in G$. Therefore, for each $n\in \mathbb{N}$ we have $n-1$ possibilities choices for gaps in $G_1\subset [m+1,2m-1] $.

Since $G$ is symmetric and $PF(G)=\{\ell_g\}$, any gap other than $\ell_g$ is a non pseudo-Frobenius number. In particular, for each  non pseudo-Frobenius number $x\in [m+2,2m-1]$ has the following property:
   \begin{equation}\label{eq1}
       \exists y\in [m+2,2m-1]\ such\ that  \ x+y=\ell_g, \ \text{where y is a non-gap}.
   \end{equation}

\begin{claim}\label{claim1}
   For each $x\in [m+2,2m-1]=[2n+2,4n-1]$ there exists a only $y\in [m+2,2m-1]$ satisfying equation and the number of entries for x is $n-1$.
\end{claim}
 In fact, since $x\in [m+2,2m-1]$ we can write it as $ x=2n+2+i$ where $i=0,\ldots,2n-3$ then $6n+1=\ell_g=x+y=2n+2+i+y \Rightarrow y=4n-1-i$. Since $0\leq i\leq 2n-3$ then $ 2n+2\leq y\leq 4n-1$, i.e., $y\in [m+2,2m-1]$.

In the next, we define the following set: $$G'=[1,m-1]\cup G'_1\cup \{2m+1\}\cup \{3m+1\} $$ where $G'_1:=\{m+1\}\cup X$,  $m=2n$ and $X\subset [m+2,2m-1]$ consists by all elements $ x\in [m+2,2m-1]$ with the following property:

\begin{equation}\label{eqX}
  x\in X \Leftrightarrow \ell_g-x \not\in X.  
\end{equation}

Note by Claim \ref{claim1} and by construction  of the $G'_1$ has $n$ elementos. Therefore, $\#G'=3n+1.$ We have $2^{n-1}$ possibilities to build $G_1'$ and therefore for $G'$. Indeed, the number of possibilities for $G_1'$ defined by the property (\ref{eqX}) is
\begin{equation*}
\frac{\prod_{i=1}^{n-1}(2n-2i)}{(n-1)!}=2^{n-1}.    
\end{equation*}
Note that any $G$ symmetric pure $(2n)-$sparse  gapset with genus $g=3n+1$ is like $G'$.

\begin{claim}\label{claim2}
     $G'$ is a gapset. 
\end{claim}
Indeed, First note that by construction $m,2m,3m$ and $4m\not\in G$. Therefore, $G$ is $m-$set. Let $z\in G$ and write $z=x+y$.
We consider three cases as follows:
\begin{enumerate}
\item If $z\in [1,m-1]$. In this case, both $x,y\in G$.

\item If $z\in G'_1$,  $    \subseteq [m+2, 2m-1]$. In this case, $2x \leq x + y = z \leq2m -1$. Thus, $x \leq m-1$ and $x \in G'_0$.

\item If $z=2m+1$. Note that $ x,y\geq m+1 \Rightarrow x+y\geq 2m+2$, so we can conclude without loss of generality that $ x\leq m$.
\begin{itemize}
    \item If $x=m$, then $y=m+1\in G'_1$.

    \item If $x\leq m-1$, then $x\in G_0'$.
    
\end{itemize}

\item If $z=3m+1$. Note that $ x,y\geq 2m \Rightarrow x+y\geq 4m$ , so we can conclude without loss of generality that  $ x\leq 2m-1$.

\begin{itemize}
    \item If  $x\leq m-1$, then $x\in G_0'$.

    \item If  $x=m$, then $y=2m+1\in G'$

    \item If $x\in [m+2,2m-1]$. Since $y=\ell_g-x $ we have by Claim \ref{claim1} and by Property (\ref{eqX}) that $x\in G \Leftrightarrow y\not\in G$ and conversely $x\not\in G \Leftrightarrow y\in G$ .
    
\end{itemize}
Therefore, $G'$ is a gapset.
\end{enumerate}



\end{proof}

\section{On pure $(2n+1)-$sparse gapsets with genus $3n+2$  }\label{sec4}

In this section we let consider the family of the pure $\k-$sparse gapsets with genus $g$ with $g=3n+2$ and $\k=2n+1$, for all $n\in \mathbb{N}$, denoted by $ \mathcal{G}_{2n+1}(g)$. More precisely, we  study the structure of this gapsets, some properties and calculate the quantity of the symmetric gapsets in $ \mathcal{G}_{2n+1}(g)$.



The next proposition tells us that the set $\{i : \ell_{i+1} -\ell_i = 2n+1\}$ has an unique element.

\begin{proposition}\label{unique g+1}
    Let $G $ be a pure $(2n+1)-$sparse  gapset with genus $g$. Then
there is an unique $ \alpha\in [1,g-1] $ such that $ \ell_{\alpha+1} -\ell_\alpha=2n +1  $,  for all $n\in \mathbb{N}$ with $n\geq 2$.
\end{proposition}
\begin{proof}
     Suppose there are $ \alpha, \beta\in [1,g-1]$ distinct such that $  \ell_{\alpha+1} -\ell_\alpha=2n$ and  $\ell_{\beta+1} -\ell_\beta=2n+1$. Then, in $[\ell_{\alpha},\ell_{\alpha+1}]$ and $[\ell_{\beta},\ell_{\beta+1}]$ we have $4n$ non-gaps.  Then, in $[1,2g]$ we have $3n+2$ gaps and at least $4n$ non-gaps. Therefore, $7n+2=4n +3n+2<\# [1,2g-1]=2g-1=6n+3$ what is a contradiction.
\end{proof}

The next result ensures that $\ell_\alpha$ belongs to either \(G_0\) or \(G_1\).

\begin{proposition}\label{ell2 alfa 2m-1}
    Let $n$ and $g =3n+2$ be positive integers  and consider
$G = \{\ell_1 < \ell_2 < \cdots < \ell_\alpha < \ell_{\alpha+1} < \cdots < \ell_g\}$ be a pure $(2n+1)-$sparse gapset with genus $g$, multiplicity m and depth $q\leq 3$, where $ \alpha = max\{i :\ell_{i+1} -\ell_{i} = 2n+1\}$. Then $\ell_\alpha  \leq 2m - 1$.
\end{proposition}
\begin{proof}
    Suppose that $\ell_\alpha \geq 2m+1$. Then, $ \ell_{\alpha+1}=\ell_\alpha+(2n+1)\geq 2m+1+2n+1\geq 3(2n+1)+1=6n+4=2g$, which is a contradiction.
\end{proof}

\begin{example}
 The gapset $ G=  [ 1, 2, 3, 4, 5, 6, 7, 8, 10, 11, 13, 14, 17, 26 ]$ has genus $14$ and is a pure $9-$sparse with multiplicity $m=9$. Therefore, $G$ is a  pseudo-symmetric gapset with genus $g=3n+2$, where $n=4$.
\end{example}
The next result gives us information about the multiplicity for a  pseudo-symmetric gapset  with genus $3n+2$ and pure $(2n+1)-$sparse .
\begin{proposition} \label{m=2n+1}
     Let $G $ be a pure $(2n+1)-$sparse  gapset with genus $g$ and multiplicity $m$. If $G$ is pseudo-symmetric, then $m=2n+1$.
\end{proposition}
\begin{proof}
    Suppose that $G$ pseudo-symmetric.    By Theorem \ref{Corollary 4.16} we have $PF(G)=\{\ell_g,\frac{\ell_g}{2}\}$, then $\frac{\ell_g}{2}=3n+1$. Note that $ \ell_g-\frac{\ell_g}{2}=3n+1>2n+1$, then exists $\ell\in G$ sush that $\frac{\ell_g}{2}<\ell<  \ell_g$. Taking $\ell=\ell_{g-1}$ exists a non-gap $s\neq 0$ such that $\ell_{g-1}+s\in G$, then $\ell_{g-1}+s=\ell_{g}$, i.e., $ \ell_{g}-\ell_{g-1}=s$. By Proposition \ref{km} we have $ 2n+1\leq m$ and we conclude that 
 $$2n+1\geq \ell_{g}-\ell=s\geq m\geq 2n+1.$$
\end{proof}

The next proposition ensures that the family $ \mathcal{G}_{2n+1}(g)$ does not have any symmetric gapsets.

\begin{proposition}\label{nao tem  simetrico}
If $G $ is a pure $(2n+1)-$sparse  gapset with genus $g$, then  G not is symmetric.
\end{proposition}
\begin{proof}
 Suppose that $G$ is symmetric. By Theorem \ref{Corollary 4.11} we hare  $PF(G)=\{\ell_g\}$, then exist $s\neq 0 $ non-gap such that $\ell_{g-1}+s=\ell_g$. Therefore,
 $$s= \ell_g-\ell_{g-1}\leq 2n+1=  m\leq s.$$
Then, we obtain that $\ell_{g-1}= 4n+2=2m$ what is a contradiction.
\end{proof}

\begin{corollary}\label{pseudo-symmetric então q=3}
     Let $G $ be a pure $(2n+1)-$sparse  gapset with genus $g$ and depth $q $. Then $ q\leq 3$. In particular, if $G$ is pseudo-symmetric, then $q=3$
\end{corollary}
\begin{proof}
    By Proposition \ref{km} we have $2n+1\leq m$, so $ \frac{1}{m}\leq \frac{1}{2n+1}$. By Proposition \ref{nao tem  simetrico} we have $c:=c(G)\leq 2g-1=6n+3$. Then
    $$ \displaystyle  q= \lceil \frac{c}{m} \rceil\leq \lceil \frac{6n+3}{2n+1} \rceil=3.$$

    In particular, if $\ell_g=2g-2$ we have that
   $$ \displaystyle  q= \lceil \frac{c}{m} \rceil= \lceil \frac{6n+3}{2n+1} \rceil=3.$$
\end{proof}

\begin{remark}
The converse of the previous result is not true.  The gapset $G=\{1,2,3,4,6,7,8,13\}$ is pure $5-$sparse gapset with genus $8$ e depth $3$, however $\ell_g=2g-3$, i.e., $G$ is not pseudo-symmetric.
\end{remark}

\begin{example} For $n=1$ we have $g=3n+2=5$ and $ 2n+1=3$.  The pure 3-sparse gapsets with genus 5 are:
\begin{itemize}
    \item   $ \{1, 2, 3, 4, 7 \}$,
    \item   $ \{1, 2, 3, 6,7 \} $,
    \item   $ \{1, 2, 4, 5, 8 \} $.
    
\end{itemize}    
\end{example}

 The following result we give a characterization for all pseudo-symmetric gapsets in $ \mathcal{G}_{2n+1}(3n+2)$.

\begin{proposition}\label{carac psedo }
      Let $G $ be a pure $(2n+1)-$sparse  gapset with genus $g$ and pseudo-symmetric.  Then $G_2=\{\ell_g\}$ and $\#G_1=n+1$, where $G=G_0\cup G_1\cup G_2$. In particular, $\alpha=g-1$, $\ell_{g-1}=2m-1 $ and $\ell_{g}=3m-1 $.
\end{proposition}
\begin{proof}
Since $G $ is pseudo-symmetric we now that $PF(G)=\{\ell_g, \frac{\ell_g}{2}\}=\{6n+2,3n+1\}$ by Theorem \ref{Corollary 4.16}. Since $G $ is pseudo-symmetric we have $\ell_g=2g-2$ and by Corollary \ref{pseudo-symmetric então q=3} we have $G=G_0\cup G_1\cup G_2$. So $\ell_g-\frac{\ell_g}{2}=3n+1>2n+1 $ we have that $\ell_{g-1}$ can not be pseudo-Frobenius number, so exists a non-gap $s\neq 0$ such that $ \ell_{g-1}+s=\ell_g$, then $$ s=\ell_g -\ell_{g-1}\leq 2n+1=m\leq s$$
Therefore, $\ell_{g-1}=\ell_g-(2n+1)=6n+2-(2n+1)=2m-1. $ So we conclude that $G_2=\{\ell_g\}$, $G_1=\{\ell_1,\ldots,\ell_n,\ell_{g-1}\}$.
\end{proof}

\begin{theorem}\label{cont p-symmetric}
    The quantity of  pseudo-symmetric gapsets in $\mathcal{G}_{2n+1}(3n+2)$ is $2^{n-1}.$
\end{theorem}
\begin{proof}
 Let $G $ be a pure $(2n+1)-$sparse  gapset with genus $g$, pseudo-symmetric.  As $G $ is pseudo-symmetric we know that $\ell_g=2g-2$, by Proposition \ref{m=2n+1}  we have $m=2n+1$ and $PF(G)=\{\ell_g, \frac{\ell_g}{2}\}$ by Theorem \ref{Corollary 4.16}. By Proposition  \ref{carac psedo } we have $G=G_0\cup G_1\cup G_2$ where $G_2=\{\ell_g\}=\{3m-1\}$, $\ell_{g-1}=2m+1\in G_1$ and $\#G_1=n+1$. By Proposition  \ref{carac psedo } $2m-1\in G_1$ and the fact that $ \frac{\ell_g}{2}= 3n+1\in G_1$  we have, for each $n\in \mathbb{N}$, $n-1$ possibilities choices for gaps in $G_1\subset [m+1,2m-1] $.

Since $G$ is pseudo-symmetric and $PF(G)=\{\ell_g,\frac{\ell_g}{2}\}$, any gap other than $\ell_g$ and $\frac{\ell_g}{2}$ is a non pseudo-Frobenius number. In particular, for each  non pseudo-Frobenius number $x\in [m+2,2m-1]$ has the following property:
 
   \begin{equation}\label{eq3}
       x,y\in [m+1,2m-2]\setminus \{\frac{\ell_g}{2}\}\ such\ that  \ x+y=\ell_g
   \end{equation}

\begin{claim}\label{claim3}
   For each $x\in [m+1,2m-2]=[2n+2,4n]$ there exists $y\in [m+1,2m-2]$ satisfying equation and the number of entries for x is $n-1$.
\end{claim}
 \begin{proof}
     In fact, since $x\in [m+2,2m-2]$ we can write it as $ x=2n+ 2+i$ where $i=0,\ldots,2n-2$ then $6n+2=\ell_g=x+y=2n+2+i+y \Rightarrow y=4n-i$. Since $0\leq i\leq 2n-2$ then $ 2n+2\leq y\leq 4n$, i.e., $y\in [m+1,2m-2]$.
 \end{proof}

In the next, we define the following set:
\begin{equation}\label{eqG'}
    G'=[1,m-1]\cup G'_1\cup \{3m-1\} ,
\end{equation}
where $G'_1:=X\cup \{\frac{\ell_g}{2}, 2m-1\}$,  $m=2n+1$ and $X\subset [m+1,2m-2]\setminus \{\frac{\ell_g}{2}\}$ with the following property:

\begin{equation}\label{eqXx}
  x\in X \Leftrightarrow \ell_g-x \not\in X.  
\end{equation}

Note by Claim \ref{claim3} and by construction  of the $G'_1$ has $n+1$ elementos. Therefore, $\#G'=3n+1.$ We have $2^{n-1}$ possibilities to build $G_1'$ and therefore for $G'$. Indeed, the number of possibilities for $G_1'$ defined by the property (\ref{eqXx}) is
\begin{equation*}
\frac{\prod_{i=1}^{n-1}(2n-2i)}{(n-1)!}=2^{n-1}.    
\end{equation*}
Note that any pseudo-symmetric pure $(2n+1)-$sparse gapset with genus $g=3n+2$ is like $G'$.

\begin{claim}\label{claim4}
     $G'$ is a gapset. 
\end{claim}
In fact, first we note that by construction $m,2m,3m\not\in G$. Therefore, $G$ is $m-$set. Let $z\in G$ and write $z=x+y$.
We consider three cases as follows:
\begin{enumerate}
\item If $z\in [1,m-1]$. In this case, both $x,y\in G$.

\item If $z\in G'_1$,  $    \subseteq [m+1, 2m-1]$. In this case, $2x \leq x + y = z \leq2m -1$. Thus, $x \leq m-1$ and $x \in G'_0$.


    

\item If $z=3m-1$. Note that $ x,y\geq 2m \Rightarrow x+y\geq 4m$ , so we can conclude without loss of generality that  $ x\leq 2m-1$.

\begin{itemize}
    \item If  $x\leq m-1$, then $x\in G_0'$.

    \item If  $x=m$, then $y=2m-1\in G'$

    \item If $x\in [m+1,2m-2]$. Since $y=\ell_g-x $ we have by Claim \ref{claim3} and by property (\ref{eqXx}) that $x\in G' \Leftrightarrow y\not\in G'$ and conversely $x\not\in G' \Leftrightarrow y\in G'$ .
    
\end{itemize}
Therefore, $G'$ is a gapset.
\end{enumerate}

\end{proof}

    
    

\section{The cardinality of the $\mathcal{G}_{2n}(3n+1)$ and $\mathcal{G}_{2n+1}(3n+2)$ }\label{sec5}

We remember that we denote by $\mathcal{M}_\k(g)$ the set of all $m$-sets with $g$ elements, such that the maximum distance between two consecutive elements (with respect to the natural order) is $\k$ and that lies on $[1,2g - 1]$. 
 If $c$ is the largest element of an $m-$set $M$, then
 its depth is defined as $\lceil\frac{c}{m} \rceil$, denoted by $q(M)$.
We denote by $\mathcal{C}_{\k}(g)$ the set of all subsets of $[1,2g-1]$ that have $g$ elements and such that the maximum distance between two consecutive elements (with respect to the natural order) is $\k$.  Let $\mathcal{G}_\k(g,q(G)\leq q):=\{G\in 
 \mathcal{G}_\k(g)|\ q(G)\leq q\}$ be the set of pure $\k-$sparse gapsets of genus g and depth at most $q$.

In this section, we introduce the map $\sigma$ and we discuss which properties the set $\sigma(G)$ has, when $G \in \mathcal{G}_{2n}(3n+1)$. Our aim is that $\sigma(G)$ is a pure $(2n+1)$-sparse gapset with genus $3n+2$. 


Let $n$ be a non-negative integer and $g=3n+1$. We define $\sigma: \mathcal{G}_{2n}(g) \to \mathcal{C}_{2n+1}(g+1)$, with 
\begin{align*}
    G=&\{1=\ell_1<\ell_2<\cdots<\ell_\alpha<\ell_{\alpha+1}<\cdots<\ell_g\} \mapsto \\
    \sigma(G)=&\{1=\ell_0< 2=\ell_1+1<\ell_2+1<\cdots<\ell_\alpha+1<\ell_{\alpha+1}+2<\cdots<\ell_g+2\},
\end{align*}
where $ \alpha=max\{i|\ \ell_{i+1}-\ell_{i}=2n\}$.

Throughout the section,   $n$ denote a  positive integer and $g=3n+1$.

\begin{theorem}\label{fi injetiva}
Let $n$ and $q$ be  non-negative integers.  The function $\sigma: \mathcal{G}_{2n}(g, q \leq 3) \to \mathcal{C}_{2n+1}(g+1)$ is well defined and injective. 
\begin{enumerate}
\item[i)] If $G \in \mathcal{G}_{2n}(g, 2)$ with multiplicity $m$, then $\sigma(G)$ is an $(m + 1)-$set of depth 2;

\item[ii)] If $G \in \mathcal{G}_{2n}(g,3)$ with multiplicity $m$, then $\sigma(G)$ is an $(m + 1)-$set of depth 3.
\end{enumerate}
\end{theorem}
\begin{proof}
By construction of the $\sigma(G)$ we have that  the maximum distance between
two consecutive elements in $\sigma(G)$ is $(\ell_{\alpha+1} + 2) - (\ell_\alpha + 1) = 2n + 1$ and there are $g + 1$ elements in $\sigma(G)$.

 It remains to solve the cases $q = 2$ and $q = 3$. Let $G=\{1=\ell_1<\ell_2<\cdots<\ell_\alpha<\ell_{\alpha+1}<\cdots<\ell_g\}$  with $m(G) = m$; in particular, $[1, m - 1] \subset G$ and $m \not\in G$ implies that $[1,m]\subset \sigma(G)$ and $m+1\not\in \sigma(G).$

 If $q=2$ , then $ m + 1 \leq\ell_g \leq 2m - 1$ and the maximum element of $\sigma(G)$, $\ell_g + 2$, is such that $(m + 1) + 1 < m + 3\leq \ell_g + 2 \leq 2m + 1 = 2(m + 1) - 1$. Thus, $2(m + 1) \not\in \sigma(G)$, which guarantees that $\sigma(G)$ is an $(m + 1)-$set with depth 2.

  If $q=3$,  we just need to prove that $2(m+1),3(m+1) \not\in \sigma(G)$. If $2(m+1) =2m+2\in \sigma(G)$ then $ 2m\in G$ with $ \ell_\alpha<2m$ or  $ 2m+1\in G$ with $ \ell_\alpha>2m+1$. But, the first condition never happens and the Proposition \ref{ell alfa 2m-1} ensures that  $\ell_\alpha\leq 2m-1 $ the second condition is also not met. Thus, $2(m + 1)  \not\in \sigma(G)$.  Now notice that $2m + 1  \leq \ell_g   \leq 3m - 1 $ and the maximum element of $ \sigma(G)$, namely $\ell_g  + 2$, is such that $2(m + 1) + 1 = 2m + 3  \leq \ell_g  + 2  \leq 3m + 1 < 3(m + 1) - 1$. Thus, $3(m + 1)  \not\in \sigma(G)$,  which guarantees that $\sigma(G)$ is an $(m + 1)-$set with depth 3.

Now we prove that $\sigma$ is injective. Let $G' = \{a_1 < a_2 < \cdots < a_{\alpha} < a_{\alpha+1} < \cdots < a_g\}$ and $\widetilde{G} = \{b_1 < b_2 < \cdots < b_{\beta} < b_{\beta+1} < \cdots <
b_g\} \in  \mathcal{G}_{2n}(g)$ such that $\sigma(G') =\sigma(\widetilde{G})$, where $\alpha = max\{i : a_{i+1} - a_{i} =2n\}$ 
and $\beta = max\{i :b_{i+1} - b_{i}  = 2n\}$. First, we prove that $\alpha = \beta$. 
Otherwise, without loss of generality, suppose that $\alpha < \beta$.
Then $a_s = b_s$, for all $s \in  [1, \alpha]\cup[\beta+1, g]$ and $a_{\beta}+2 = b_{\beta}+1$.  Then, $ 2n = b_{\beta+1}-b_{\beta} = a_{\beta+1}-(a_{\beta}+1)$, i.e., $ 2n+1 = a_{\beta+1}-a_{\beta}$ 
which is a contradiction. Hence, $\alpha =\beta$ and $ G' = \widetilde{G}$.

\end{proof}

Now, we look at pure $(2n)-$sparse gapsets with genus $g=3n+1$ and depth $2$ and $3$, namely
$\mathcal{G}_{2n} (g, 2):=\{G \in \mathcal{G}_{2n} (g) :\ q(G) = 2\}$ and $\mathcal{G}_{2n} (g, 3):=\{G \in \mathcal{G}_{2n}(g) :\ q(G) = 3\}$,
respectively. Theorem \ref{fi injetiva} guarantees that gapsets with multiplicity $m$ and depth at most $3$ are mapped onto $(m + 1)-$sets with depth at most 3. Now we look for conditions that ensure that those $ (m + 1)-$sets are in fact gapsets.

\begin{proposition}\label{Gq2}
 Let $n$ and $g=3n+1$ be non-negative integers and $\sigma_2:=\sigma|\mathcal{G}_{2n} (g,2)$. Then
$Im(\sigma_2) \subset \mathcal{G}_{2n+1}(g + 1, 2)$.
\end{proposition}
\begin{proof}
    Let $ G \in  Im(\sigma_2)$. Then there is $G' \in\mathcal{G}_{2n} (g)$ such that $\sigma_2(G') = G$. If $m(G') = m$ and $q(G') = 2$, then Theorem \ref{fi injetiva} guarantees that G  is an $(m + 1)-$set with depth 2. By Proposition \ref{q=2}, we conclude that $G$ is a gapset and $G \in \mathcal{G}_{2n+1} (g + 1, 2)$.
\end{proof}

\begin{proposition}\label{Gq3}
     Let $n$ and $g=3n+1$ be non-negative integers and $\sigma_3:=\sigma|\mathcal{G}_{2n} (g,3)$. Then
$Im(\sigma_3) \subset \mathcal{G}_{2n+1}(g + 1, 3)$.   
\end{proposition}
\begin{proof}
Let $G = \{\l_1 < \l_2 < \cdots < \l_\a < \l_{\a+1}  < \cdots < \l_g\} \in \mathcal{G}_{2n}(g,3)$, with $m(G) = m$. By Theorem \ref{fi injetiva}, $\sigma(G)$ is a $m$-set that belongs in $\mathcal{M}_{2n+1}(g+1,3)$. It remains to prove that $\sigma_3(G)$ is a gapset. Consider the canonical partition of $\sigma(G)$, namely $G'_0 \cup G'_1 \cup G'_2$. Let $z \in \sigma(G)$ and write $z = x + y$, with $x \leq y$. We consider three cases as follows:

\begin{enumerate}
\item  $z \in G'_0 = [1, m]$. In this case, both $x$ and $y \in G'_0$. 
\item  $z \in G'_1 \subseteq [m+2, 2m+1]$. In this case, $2x \leq x + y = z \leq2m +1$. Thus, $x \leq m$ and $x \in G'_0$.
  \item $z \in G'_2 \subseteq [2m+3, 3m+1]$. We claim that if $x \leq m$ or $y \geq 2m+1$, then $x \in G'_0$. The first case is trivial and the second one implies that $3m + 1 \geq z \geq x + 2m + 1$ and we obtain $x \leq m$. Hence it remains to show that if $x, y \in [m+1, 2m]$, with $x \leq y$, then $x$ or $y \in G'_1$. 
  \begin{itemize}
  \item Consider $\l_{\a+1} \leq 2m - 1$. In this case, we claim that $y-1 \leq \l_\a$. In fact, if $y \geq \l_\a + 2$, then $z \geq (m + 1) + (\l_\a + 2) = \l_\a + m + 3$. Thus, $G \ni z-2 \geq \l_\a + m + 1$, which is a contradiction according to Proposition \ref{abaixo}. Thus, both $x-1$ and $y-1 \leq \l_\a$ and we can write $z - 2 = (x-1) + (y-1)$. Since $z-2 \in G$ and $G$ is a gapset, we conclude that $x-1$ or $y-1 \in G$ and thus $x \in \sigma(G)$ or $y \in \sigma(G)$. 
   \item  Consider $\l_{\a+1} \geq 2m + 1$. By Proposition \ref{ell alfa 2m-1}, we have $\l_\a \leq 2m - 1$ and thus $\l_\a + 1 \in G'_1$. Hence, $z \geq \l_{\a+1}+2$. In this case, $x \leq y \leq 2m < \l_{\a+1}$ and $G \ni z - 2 = (x-1) + (y-1)$. Since $G$ is a gapset, then $x-1$ or $y-1 \in G$. Hence, $x \in \sigma(G)$ or $y \in \sigma(G)$.
  \end{itemize}    
  
\end{enumerate}

\end{proof}

The next result show that follow union $\mathcal{G}_{2n+1}(3n+2)= \mathcal{G}_{2n+1}(3n+2,q \leq 3)\cup \{G\in \mathcal{G}_{2n+1}(g+1)|  G \ \text{is pseudo-symmetric}\}$
is disjoint.

\begin{proposition}\label{g'}
For all $G'\in \sigma( \mathcal{G}_{2n}(g,q \leq 3))$   we have $\ell_{g'}\leq 2g'-3 $ where $ g'=g(G')=3n+2$ and $g=3n+1$. In particular, $G'$ is not pseudo-symmetric.
\end{proposition}
\begin{proof}
   If $G'\in \sigma( \mathcal{G}_{2n}(g,q))$ exists $G\in \mathcal{G}_{2n}(g,q) $ such that $G'= \sigma(G)$. By  Theorem \ref{nao tem pseudo simetrico}  we have $\ell_g\leq 2g-3=6n-1$ and definition of the $\sigma$  we have $\ell_{g'}=\ell_g+2\leq 6n+1=2g'-3$
\end{proof}

The next theorem guarantees that the image of $\mathcal{G}_{2n}(g,q\leq 3)$ by $\sigma$ does not contain pseudo-symmetric gapsets. This is important because it provides insight into the structure of the set $\sigma(\mathcal{G}_{2n}(g,q\leq 3))$.

\begin{theorem}\label{main1}
The set   $\sigma( \mathcal{G}_{2n}(g,q\leq 3))$ no has a pseudo-symmetric pure $(2n+1)-$sparse gapset  with genus $g+1$. In particular,  $\sigma( \mathcal{G}_{2n}(g,q\leq 3))=\mathcal{G}_{2n+1}(g+1)\setminus \mathcal{P}$, where $\mathcal{P}=\{G\in \mathcal{G}_{2n+1}(g+1)|\  G\ \text{pseudo-symmetric}\}$.
\end{theorem}
\begin{proof}
The Theorem \ref{fi injetiva} guarantees that $\sigma$ is injective; Propositions \ref{Gq2} and \ref{Gq3} guarantee that $\sigma$ is well defined. Hence, we only have to prove that $\sigma$ is surjective. Now, notice that  that $\mathcal{G}_{2n}(g,q \leq 3)$ has only the gapsets with depth 2 or 3. The Proposition \ref{g'} ensures that $\sigma( \mathcal{G}_{2n}(g,q \leq 3))$ no has a $G $  pure $(2n+1)-$sparse  gapset with genus $g=3n+2$, pseudo-symmetric.

Write $G = \{\l_1 < \l_2 < \l_3 < \cdots < \l_\a < \l_{\a+1}  < \ldots < \l_{g+1}\}\in \mathcal{G}_{2n+1}(g+1)\setminus \mathcal{P}$ and assume that its multiplicity is $m$. Now $ \l_{g+1}\leq 2(g+1)-3$ by Proposition \ref{g'}. In particular $G$ has depth at most 3 and thus $\l_{g+1} \leq 3m - 1$. The Proposition \ref{unique g+1} ensures that $\a$ is the unique index such that $\l_{\a+1} - \l_{\a} = 2n+1$. We show that $G' = \{\l_2 - 1 < \l_3 - 1 < \ldots < \l_\a - 1 < \l_{\a+1} - 2 < \ldots < \l_{g+1} - 2\} \in \mathcal{G}_{2n}(g)$ and satisfies $\sigma(G') = G$. 

Naturally, $\#G' = g$, the maximum distance between to elements of $G'$ is $(\l_{\a+1} - 2) - (\l_\a - 1) = 2n$ and definition of the $G'$ the first positive integer that does not belong to $G'$ is $m - 1$.  Now we prove that $G'$ is a $(m-1)$-set:
\begin{itemize}
  \item $2(m-1) \notin G'$  \\
Suppose that $2m - 2 \in G'$. In this case, we could have $2m \in G$ or $2m - 1 \in G$. The first case does not occur, because $m$ is the multiplicity of $G$. In the second case, we have  $2m - 1 \leq \l_\a$ (otherwise, the correspondent element would be $(2m - 1) - 2 = 2m - 3)$. By Proposition \ref{ell2 alfa 2m-1}, we have that $\l_\a = 2m - 1$. Thus, $\l_{\a+1} = \l_\a + (2n + 1) = 2m - 1 + 2n + 1 \geq 2(2n +1) + 2n  =6n+2 = 2(g+1)-2$, which contradicts the Proposition \ref{nao tem  simetrico}. Hence, $2m - 2 \notin G'$.

  \item $3(m-1) \notin G'$ \\
We already know that $\l_{g+1} \leq 3m - 1$ and thus the Frobenius number of $G'$ is such that $\l_{g+1} - 2 \leq 3m - 3$. Hence, we only have to show that $\l_{g+1} \neq 3m - 1$ which implies that $\l_{g+1} - 2 \leq 3m - 4$. Suppose that $\l_{g+1} = 3m - 1$. In this case, we have $2m - 1 \in G$ and we could have $\l_\a < 2m - 1$ or $\l_\a = 2m - 1$. The first case implies that $\l_{\a}+m < 3m - 1$. By Proposition \ref{abaixo} $\ell_{g+1}\leq \ell_\a +m$ and we would have $3m-1=\ell_{g+1}<3m-1 $
what is absurd.
In the second case, we obtain $\l_{\a+1} = \l_\a + (2n+ 1) = 2m - 1 + 2n + 1 \geq 2(2n+1) + 2n = 3(2n) + 2 =6n+2 = 2(g+1)-2$, which contradicts the definition of $\mathcal{P}$ and the Proposition \ref{nao tem  simetrico}. Hence, $3m - 3 \notin G'$.

\end{itemize}

Therefore, $G'$ is a $(m-1)$-set for $q=2$ and $q=3$. Hence, we can write $G' = G'_0 \cup G'_1 \cup G'_2$, with $G'_0 = [1,m-2]$, $G'_1 \subseteq [m, 2m-3]$ and $G'_2 \subseteq [2m-1, 3m-4]$. Now we prove that $G'$ is a gapset. Let $z \in G'$ and write $z = x + y $, with $x, y \in \N$ and $x \leq y$.

\begin{itemize}
  \item If $z \in G'_0$, then $x$ and $y \in G'_0$. 
  \item If $z \in G'_1$, then $z \leq 2m - 3$ and $2x \leq 2m - 3$ and $x \in G'_0$.
  \item Let $z \in G'_2$. If $x \leq m - 2$, then $x \in G'_0$. If $y \geq 2m-1$, then $x \leq m - 3$ and $x \in G'_0$. Hence, we can consider $x, y \in [m, 2m - 3]$. Notice that $\l_\a \leq 2m - 2$ (if $\l_\a = 2m - 1$, then $2m - 2 \in G'$, which does not occur). In particular, $\l_\a - 1 \leq 2m - 3$ and $z \geq \l_{\a+1} - 2$. Thus $z = \l_t - 2$, for some $t \in [\a+1, g+1]$. We claim that $y < \l_\a$. Otherwise, $z = x + y \geq m + \l_\a$ and it implies that $\l_t > \l_\a + m$, which is a contradiction according to Proposition \ref{abaixo}. Hence, $x \leq y < \l_\a$. By writing $\l_t - 2 = z = x + y$, i.e., $\l_t = (x+1) + (y+1)$ and using that $G$ is a gapset, we conclude that $x+1 \in G$ or $y+1 \in G$. Hence, $x \in G'$ or $y \in G'$.
\end{itemize}

\end{proof}

The family of gapsets $\mathcal{G}_{2n}(3n+1)$ is composed of gapsets of depth $q\in \{2,3,4\}$. Since the Theorem \ref{Gsimetrico equiv q=4} guarantees that the only gapsets with $q=4$ are the symmetric ones, we can write $\mathcal{G}_{2n}(3n+1)=\mathcal{G}_{2n}(3n+1,q\leq 3)\cup\mathcal{G}_{2n}(3n+1,4) $ for   $q\in \{2,3\}$. On the other hand, the family of gapsets $\mathcal{G}_{2n+1}(3n+2)$ is composed of gapsets of depth 2 or 3. According to the Proposition \ref{nao tem  simetrico} a familia $\mathcal{G}_{2n+1}(3n+2)$ does not contain symmetric gapsets and the Corollary \ref{pseudo-symmetric então q=3} guarantees that pseudo-symmetric gapsets have depth 3. Therefore, we can write $\mathcal{G}_{2n+1}(3n+2)=\{G\in \mathcal{G}_{2n+1}(3n+2)|\ G\ \text{is not pseudo-symmetric}\}\cup \{G\in \mathcal{G}_{2n+1}(3n+2)|\ G\ \text{is pseudo-symmetric}\}$.
The Theorem \ref{main1} guarantees that the cardinality of the family $\mathcal{G}_{2n}(3n+1,q\leq 3)$ is the same as that of the set $\{G\in \mathcal{G}_{2n+1}(3n+2)|\ G\ \text{is not pseudo-symmetric}\}$.

\begin{corollary}\label{main2}
    $\#\mathcal{G}_{2n}(3n+1) =\#\mathcal{G}_{2n+1}(3n+2)$ for all $n\in \mathbb{N}$.
\end{corollary}
\begin{proof}
We have by Theorem \ref{contagem G simetrico q=4} that $\#\mathcal{G}_{2n}(3n+1,4)=2^{n-1}$ and Theorem \ref{cont p-symmetric} guarantees that $\#\{G\in \mathcal{G}_{2n+1}(3n+2)|\ G\ \text{is pseudo-symmetric}\}= 2^{n-1}$. Now, from Theorem \ref{main1} we conclude that $\#\mathcal{G}_{2n}(3n+1) =\#\mathcal{G}_{2n+1}(3n+2)$ for all $n\in \mathbb{N}$.
\end{proof}

\begin{table}\label{tab1}
\centering
\caption{Some informations related to $(s_n)$}
\begin{tabular}{c c c c}
\hline
$n$ & $s_n$ & $\dfrac{s_{n}}{s_{n-1}}$ & $\frac{ \displaystyle\sum_{i=0}^{n} s_{i}}{s_{n}}$ 
\\
\noalign{\smallskip}\hline\noalign{\smallskip}
1 & 3 & 2.6666 & 1 \\
\hline
2 & 8 & 2.7500 & 1.3750 \\
\hline
3 & 22 & 2.4545 & 1.500 \\
\hline
4 & 54 & 2.500 & 1.6111 \\
\hline
5 & 135 & 2.4518 & 1.6444 \\
\hline
6 & 331 &  2.4410 & 1.6706\\
\hline
 7 & 808 &  & 1.6844 \\
\noalign{\smallskip}\hline
   \end{tabular}
\end{table}

Let $s_n$ is the cardinality of $\mathcal{G}_{2n}(3n+1)$ with $n \in \mathbb{N}$. Here are the first few values of the sequence $(s_n)_n$:    $(s_1,s_2,s_3,s_4,s_5,s_6,s_7)=(3,8,22,54,135,331,808)$.

We can notice that the sequence $(s_n)_n$ has an asymptotic growth similar to the sequence $(g_w)$ described in \cite{GM}.

 We conclude this section by presenting Table \ref{g,k}, which illustrates our results. The values  were obtained using the package \cite{GAP} in GAP.

\begin{landscape}
\begin{table}
\caption{A few values for $\#\mathcal{G}_{\k}(g)$}
\label{g,k}
\begin{tabular}{|c||c c c c c c c c c c c c c c c c c c c c|| c|}
 \hline
\diagbox[height=0.6cm]{$g$}{$\k$} & $0$ & $1$ & $2$ & $3$ &$4$ & $5$ & $6$ & $7$ & 8 & 9 & 10 & 11 & 12 & 13 & 14 & 15 & 16 & 17 & 18 & 19 & $n_g$ \\
\hline
0 & 1 &  &  &  &  &  &  &  &  &  &  &  &  &  &  &  &  &  &  &  & 1 \\ 
\hline
1 &  & 1 &  &  &  &  &  &  &  &  &  &  &  &  &  &  &  &  &  &  & 1 \\
\hline
2 &  & 1 & 1 &  &  &  &  &  &  &  &  &  &  &  &  &  &  &  &  &  & 2 \\
\hline
3 &  & 1 & 2 & 1  &  &  &  &  &  &  &  &  &  &  &  &  &  &  &  &  & 4 \\
\hline
4 &  & 1 & \textcolor{red}{3} & 2 & 1 &  &  &  &  &  &  &  &  &  &  &  &  &  &  &  & 7 \\
\hline
5 &  & 1 & 5 & \textcolor{red}{3} & 2 & 1 &  &  &  &  &  &  &  &  &  &  &  &  &  &  & 12 \\
\hline
6 &  & 1 & 7 & 7 &  5 & 2 & 1 &  &  &  &  &  &  &  &  &  &  &  &  &  & 23 \\
\hline
7 &  & 1 & 10 & 12 & \textcolor{red}{8} &  5 & 2 & 1 &  &  &  &  &  &  &  &  &  &  &  &  & 39 \\
\hline
8 &  & 1 & 15 & 18 & 17 & \textcolor{red}{8} & 5 & 2 & 1 &  &  &  &  &  &  &  &  &  &  &  & 67 \\
\hline
9 &  & 1 & 20 & 31 & 28 & 18 & 12 & 5 & 2 & 1 &  &  &  &  &  &  &  &  &  &  & 118 \\
\hline
10 &  & 1 & 27 & 51 & 49 & 34 & \textcolor{red}{22} & 12 & 5 & 2 & 1 &  &  &  &  &  &  &  &  &  & 204 \\
\hline
11 &  & 1 & 38 & 78 & 87 & 57 & 40 & \textcolor{red}{22} & 12 & 5 & 2 & 1 &  &  &  &  &  &  &  &  & 343 \\
\hline
12 &  & 1 & 51 & 125 & 147 & 100 & 76 & 42 & 30 & 12 & 5 &  2 & 1 &  &  &  &  &  &  &  & 592 \\
\hline
13 &  & 1 & 70 & 195 & 237 & 177 & 134 & 83 & \textcolor{red}{54} & 30 & 12 & 5 & 2 & 1 &  &  &  &   &  &   & 1001 \\
\hline
14 &  & 1 & 95 & 297 & 399 & 309 & 239 & 150 & 99 & \textcolor{red}{54} & 30 & 12 & 5 & 2 & 1 &  &  &  &  &  & 1693 \\
\hline
15 &  & 1 & 128 & 457 & 654 & 530 & 422 & 259 & 183 & 103 & 70 & 30 & 12 & 5 & 2 & 1 &  &  &  &  & 2857 \\
\hline
16 &  & 1 & 172 & 705 & 1061 & 902 & 723 & 452 & 336 & 199 & \textcolor{red}{135} & 70 & 30 & 12 & 5 & 2 & 1 &  &  &  & 4806 \\
\hline
17 &  & 1 & 230 & 1074 & 1717 & 1513 & 1248 & 811 & 590 & 363 & 243 & \textcolor{red}{135} & 70 & 30 & 12 & 5 & 2 & 1 &  &  & 8045 \\
\hline
18 &  & 1 & 309 & 1621 & 2777 & 2535 & 2148 & 1411 & 1037 & 646 & 444 & 251 & 167 & 70 & 30 & 12 & 5 & 2 & 1 &  & 13467 \\
\hline
19	  & & 1 & 413 & 2448 & 4464 & 4232 & 3636 & 2434 & 1810 & 1124 & 804 & 480 & \textcolor{red}{331} & 167 & 70 & 30 & 12 & 5 & 2 & 1 & 22464 \\
\hline
\end{tabular}
\end{table}
\end{landscape}

\section{Further questions}\label{sec6}


The values obtained in Table \ref{g,k} indicate that it is possible for an injective function to exist between the sets $\mathcal{G}_\k(g)$ and $\mathcal{G}_{\k+1}(g+1)$, for example, in \cite{GM} they use the condition $2g\leq 3k$ to build this map. Furthermore, these values also indicate that it is possible to obtain an injective map from $\mathcal{G}_\k(g)$ to $\mathcal{G}_{\k}(g+1)$. We can write these questions as follows:

\begin{question}
Let $g$ and $\k$ be non-negative integers. Is there an injective map \linebreak $\mathcal{G}_\k(g) \to \mathcal{G}_{\k+1}(g+1)$?
\end{question}

\begin{question}
Let $g$ and $\k$ be non-negative integers. Is there an injective map \linebreak $\mathcal{G}_\k(g) \to \mathcal{G}_{\k}(g+1)$?
\end{question}

Notice that a positive answer to any of them implies that the sequence $(n_g)$ is increasing for $g \in \N$.

Another question that arises is about the behavior of the sequence $(s_n)_n$, where $s_n = \#\mathcal{G}_{2n}(3n+1)$. In fact, Theorem \ref{contagem G simetrico q=4} guarantees that for depth 4 a part of this sequence is known, for all $n \in \N$. Therefore, it may be interesting to know information about its sequence, which was unknown at OEIS before this work.

\begin{question}
Is it possible to determine a closed formula for the sequence $(s_n)_n$, for each $n\in \mathbb{N}?$
\end{question}

\begin{question}
Are there remarkable properties about the sequence $(s_n)_n$?
\end{question}


\end{document}